\documentclass[reqno]{amsart}
\usepackage{amsmath, amssymb, amsthm, epsfig}
\usepackage{hyperref, latexsym}
\usepackage{url}
\usepackage[mathscr]{euscript}

\usepackage{color}
\usepackage{fullpage} 
\usepackage{setspace}

\def\today{\ifcase\month\or
	January\or February\or March\or April\or May\or June\or
	July\or August\or September\or October\or November\or December\fi
	\space\number\day, \number\year}

\newtheorem{theorem}{Theorem}

\newtheorem{proposition}[theorem]{Proposition}

\theoremstyle{definition}

\theoremstyle{remark}

\newcommand{\C}{\mathbb{C}}
\newcommand{\R}{\mathbb{R}}

\newcommand{\Z}{\mathbb{Z}}

\newcommand{\hh}{\tfrac12}

\newcommand{\im}{{\rm Im}\,}
\newcommand{\re}{{\rm Re}\,}

\begin{document}
	\title[A note on the zeros of approximations of\\ the Ramanujan $\Xi-$function]{A note on the zeros of approximations of \\the Ramanujan $\Xi-$function}
	\author[Chirre]{Andr\'{e}s Chirre and Oswaldo Vel\'asquez Casta\~n\'on}
	
	\address{Department of  Mathematical Sciences, Norwegian University of Science and Technology, NO-7491 Trondheim, Norway}
	\email{carlos.a.c.chavez@ntnu.no }
	\address{IMCA - Instituto de Matem\'aticas y Ciencias Afines, Lima, Per\'u}	
	\email{oswaldo@imca.edu.pe}


	\allowdisplaybreaks
	\numberwithin{equation}{section}
	
	\maketitle
	
	\begin{abstract} In this paper, we review the study of the distribution of the zeros of certain approximations for the Ramanujan $\Xi-$function given by H. Ki \cite{haseoki}, and we provide a new proof of his results. Our approach is motivated by the ideas of Vel\'asquez \cite{ov} in the study of the zeros of certain sums of entire functions with some condition of stability related to the Hermite-Biehler theorem.

	\end{abstract}

		
\section{Introduction}

\subsection{Background} Let $\tau(n)$ be the Ramanujan's tau-function, defined by 
	\[
	\Delta(z)=\displaystyle\sum_{n=1}^{\infty}\tau(n)q^{n}=z\prod_{n=1}^\infty(1-q^n)^{24},
	\]
	where $q=e^{2\pi iz}$, and $\im{z}>0$. It is well known that $\Delta(z)$ spans the space of cusp forms of dimension $-12$ associated with the unimodular group. The associated Dirichlet series and Euler product for $\Delta(z)$ is given by
	\[
	L(s)=\displaystyle\sum_{n=1}^{\infty}\frac{\tau(n)}{n^{s}}=\prod_{p}\big(1-\tau(p)p^{-s}+p^{11-2s}\big)^{-1},
	\]
	where the series and the product are absolutely convergent for $\re{s}>13/2$. 
	Let us define the Ramanujan $\Xi$--function, denoted by 	$\Xi_R(s)$, as follows
	\[
	\Xi_R(s)=(2\pi)^{is-6}L(-is+6)\Gamma(-is+6), 
	\]
	where $\Gamma(s)$ is the Gamma function. Another representation for $\Xi_R(s)$ is given by
	\[
	\Xi_R(s)=\displaystyle\int_{-\infty}^{\infty}\phi(t)e^{ist}\,dt,
	\]
	where
	\begin{align}  \label{20_58_25_01}
	\phi(t)=e^{-2\pi\cosh(t)}\displaystyle\prod_{k=1}^{\infty}\big[\big(1-e^{-2\pi ke^{t}}\big)\big(1-e^{-2\pi ke^{-t}}\big)\big]^{12}.
	\end{align}
In \cite{Hardy}, Hardy highlighted the importance of the location of the zeros of $\Xi_R(s)$ in the strip $|\im(s)|\leq \hh$. The Riemann hypothesis for the Ramanujan zeta function states that all zeros of $\Xi_R(s)$ are real. 

\subsection{Zeros of the approximations $\Xi_F(s)$}
The purpose of this paper is to study the distribution of the zeros of certain approximations for the Ramanujan $\Xi$--function. Inspired in the representation \eqref{20_58_25_01}, Ki \cite{haseoki} defined these approximations as follows: Let $F$ be a finite sequence of complex numbers $a_0, a_1, ..., a_n$ such that at least one of them is different from zero. We define the function
	\[
	\Xi_F(s)= \displaystyle\int_{-\infty}^{\infty}\phi _{F}(t)e^{ist}\, dt,
	\]
	where
	\[
	\phi_F(t)=e^{-2\pi\cosh{t}}\Bigg(\displaystyle\sum_{m=0}^{n}a_me^{-2\pi me^{t}}\Bigg)\Bigg(\displaystyle\sum_{m=0}^{n}\overline{a_m}e^{-2\pi me^{-t}}\Bigg).
	\]
We recall that $\overline{\Xi_F(\overline{s})}=\Xi_F(s)$, and one can see that for some sequences $F_k$, the function $\Xi_{F_k}(s)$ converges uniformly to $\Xi_R(s)$ on all compact subsets of $\C$. 

\smallskip

Throughout this paper, we will study the distribution of the zeros of the function $C_F(s):=\Xi_F(-is)$. Note that the zeros of $C_F(s)$ are symmetric respect to the line $\re{s}=0$. Using the argument principle, Ki \cite[Theorem 1]{haseoki} established  for $T\geq 2$ that\footnote{\,\,\, Throughout the paper we use the Vinogradov's notation $f=O(g)$ (or $f\ll g$) to mean that $|f(t)|\leq C|g(t)|$ for a certain constant $C>0$ and $t$ sufficiently large.}
\begin{align*} 
	N(T,C_F)=\frac{T}{\pi}\log\frac{T}{e\pi}+O(\log{T}),
	\end{align*}
where $N(T,C_F)$ stands for the number of zeros of $C_F(s)$ such that $1\leq \im{s} < T$, counting multiplicity. In the lower half-plane a similar result holds. Moreover, using the method developed by Levinson \cite{levinson}, he stated that
	\begin{align} \label{resultki}
\overline{N}(T, C_F)-\overline{N_1}(T,C_F)=O(T),
	\end{align}
where $\overline{N}(T, C_F)$ stands for the number of zeros of $C_F(s)$ such that $|\im{s}| < T$, counting multiplicity and $\overline{N_1}(T,C_F)$ denotes the number of simple zeros such that $|\im{s}|<T$ and $\re{s}=0$. In a sense, it means that almost all zeros of $C_F(s)$ lie on the line $\re{s}=0$ and are simple. Our first goal is to establish a refinement of \eqref{resultki}.

\begin{theorem} \label{16_17_3_20} For $T
\geq 2$ we have \begin{align*}
	0\leq \overline{N}(T, C_F)-\overline{N_1}(T,C_F)\leq \bigg(32n+\dfrac{32\ln(2n+1)}{\pi}\bigg)T +O(1).
	\end{align*}	
\end{theorem}

\medskip

On the other hand, Ki \cite[Theorem 2]{haseoki} a result about the vertical distribution of the zeros of $C_F(s)$, based on the zeros of the function $\psi_F(s)$, defined by   		
\begin{align} \label{17_36_2_20}
\psi_F(s)=\pi^{-s}\displaystyle\sum_{m=0}^{n}a_m(2m+1)^{-s}.
\end{align}
Let $k\geq0$ be an integer such that $P(1)=P'(1)=\cdot\cdot\cdot=P^{(k-1)}(1)=0$ and $P^{(k)}(1)\neq0$, where $P(y)=\sum_{m=0}^n{a_m}y^{m}$. 

\begin{theorem} \label{teo123}
	 Let $\Delta_{*}<\Delta_{**}$ be positive real numbers. Suppose that ${\psi_F}(s-k)$ has finitely many zeros in $-\Delta_{**}<\re{s}<\Delta_{*}$. Let $\delta$ be such that $0<\delta<\Delta_{*}$. Then all but finitely many zeros of ${C_F}(s)$ which lie in $|\re{s}|\leq \delta$ are on the line $\re{s}=0$. In particular, all but finitely many zeros of ${C_F}(s)$ are on the line  $\re{s}=0$, if ${\psi_F}(s-k)$ has finitely many zeros in $\re{s}>-\Delta_{**}$.
\end{theorem}
 
Ki included a second proof for the second part of Theorem \ref{teo123}. In particular, this second proof gave information about the simplicity of the zeros of $C_F(s)$. Anyway, Ki conjectured that second case for $\psi_F(s-k)$ is not possible. On the other hand, using \eqref{poly} is clear that ${\psi_F}(s-k)$ has the same set of zeros of a Dirichlet polynomial in the framework of \cite[Subsection 12.5]{bellman}. The set of zeros of a Dirichlet polynomial is quasi-periodic (see \cite[Appendix 6, p. 449]{levin}). Then, if $s_0=\sigma_0+i\tau_0$ is a zero of the Dirichlet polynomial, for any $\varepsilon>0$ we can construct a sequence $\{s_n=\sigma_n+i\tau_n\}_{n\in\mathbb{N}}$ of zeros, such that $\sigma_n\in]\sigma_0-\varepsilon, \sigma_0+\varepsilon[$ for all $n\in\mathbb{N}$ and $\tau_n\to \pm\infty$. This implies that each open vertical strip has no zeros or has infinite zeros. Therefore, the hypothesis in Theorem \ref{teo123} is reduced to: ${\psi_F}(s-k)$ has no zeros in $-\Delta_{**}<\re{s}<\Delta_{*}$. Our second goal in this paper is to give a new proof of this result.

\begin{theorem} \label{teo1}
	Let $\Delta_{*}<\Delta_{**}$ be positive real numbers. Suppose that ${\psi_F}(s-k)$ has no zeros in $-\Delta_{**}<\re{s}<\Delta_{*}$. Let $\delta$ be such that $0<\delta<\Delta_{*}$. Then all but finitely many zeros of ${C_F}(s)$ which lie in $|\re{s}|\leq \delta$ are on the line $\re{s}=0$ and are simple.
\end{theorem}

We highlight that our proof includes information about the simplicity of the zeros for the first case. The key relation between the functions $C_F(s)$ and $\psi_F(s-k)$ is due by de Bruijn \cite[p. 225]{bruijn}, who showed that 
\[
C_F(s)= \displaystyle\sum_{m=k}^{\infty}b_m \psi_F(s-m)\Gamma(s-m) + \displaystyle\sum_{m=k}^{\infty}\overline{b_m\hspace{0.05cm}\psi_F(-s-m)}\Gamma(-s-m),
\]
where $b_m$ are complex numbers and $b_k\neq 0$. 

\subsection{Strategy outline} Our approach is motivated by a result of Vel\'asquez \cite[Theorem 36]{ov}, about the distribution of the zeros of a function of the form $f(s)= h(s)+h^{*}(2a-s)$, where $h(s)$ is a meromorphic function\footnote{\,\,\, For a meromorphic function $h(s)$, we define the function  $h^*(s)=\overline{h({\overline{s}})}$.}, and $a\in\R$. This result can be regarded as a generalization of the necessary condition of stability for the function $h(s)$, in the Hermite–Biehler theorem \cite[21, Part III, Lecture 27]{levin}. In our case, using an auxiliary function $W_F(s)$, we have the representation $C_F(s)=h(s)+h^*(-s)$, where $h(s)=W_F(-is-i/2)$. Some estimates of $h(s)$ due by Ki \cite[Theorem 2.1]{haseoki} play an important role to establish the necessary growth conditions in \cite[Theorem 36]{ov}. On the other hand, the strong relation between the zeros of $h(s)$ and $\psi_F$(s) (see \eqref{16_14_3_20}), implies to study the distribution of zeros of $\psi_F$(s), as a set of zeros of a Dirichlet polynomial.

\smallskip

Throughout the paper, we fix a sequence $F$. For a function $f(s)$ and the parameters $\sigma_1<\sigma_2$, and $T_1<T_2$, we denote the counting function
\[
N(\sigma_0,\sigma_1,T_1,T_2,f)=\#\{s\in\mathbb{C}: f(s)=0,\, \sigma_0<\sigma<\sigma_1, \hspace{0.1cm} T_1<\tau<T_2\},
\]
\[
\widehat{N}(\sigma_0,\sigma_1,T_1,T_2,f)=\#\{s\in\mathbb{C}: f(s)=0,\, \sigma_0\leq\sigma\leq\sigma_1, \hspace{0.1cm} T_1<\tau<T_2\},
\]
where, in both cases, the counts are with multiplicity, and
\[N_{0}^{'}(T,g)=\#\{s\in\mathbb{C}: g(s)=0, \hspace{0.1cm} \re{s}=0, \hspace{0.1cm} |\im{s}|<T\},
\]
where the count is without multiplicity.

\section{Preliminaries results}

In this section we collect several results for our proof. We highlight that in \cite[Proposition 2.3]{haseoki}, Ki showed that there is a constant $\beta_0>0$ such that $C_F(s)\neq 0$, for $|\re{s}|\geq \beta_0$. This implies that for $\beta\geq \beta_0$,
\begin{align} \label{18_58_3_20}
\overline{N}(T,C_F)=N(-\beta, \beta, -T ,T , C_F).
\end{align}
Therefore, we can restrict our analysis of the zeros in vertical strips. Now, let us start to find a new representation for $C_F(s)$. We define the entire function
\[
{W_F}(s)= \displaystyle\int_{-\infty}^{\infty}{{\tilde{\phi}}} _{F}(t)e^{ist}\, dt,
\]
where
\[
{\tilde{\phi}}_{F}(t)=\frac{e^{-2\pi\cosh{t}}}{\displaystyle{e^{t/2}}+\displaystyle{e^{-t/2}}}\Bigg(\displaystyle\sum_{m=0}^{n}a_me^{-2\pi me^{t}}\Bigg)\Bigg(\displaystyle\sum_{m=0}^{n}\overline{a_m}e^{-2\pi me^{-t}}\Bigg).
\] 
Then, we obtain the following relation
\begin{align}
C_F(s)=W_F\bigg(-is-\frac{i}{2}\bigg)+W_F\bigg(-is+\frac{i}{2}\bigg). \label{relacion}
\end{align}
If we denote by
\begin{align} \label{0_45}
h(s)=W_F\displaystyle\bigg(-is-\frac{i}{2}\bigg),
\end{align}
we rewrite \eqref{relacion} as
\begin{equation*}
C_F(s)=h(s)+h^*(-s). 
\end{equation*}
This representation allows us to use the following result (see \cite[Theorem 36]{ov}).
\begin{theorem} \label{teoprin} Let $\sigma_0>0$ be a parameter and $h(s)$ be an entire function such that $h(s)\neq 0$ for $\re{s}=\sigma_0$. We define the entire function 
	$$f(s)=h(s)+h^*(-s).$$ 
Suppose that the function
	\[
	F(s)=\frac{h^*(-s)}{h(s)}
	\]
	satisfies the following conditions.
\\
\noindent{\rm (i)} $F(s)\neq \pm 1$ on the line $\re{s}=\sigma_0$, and for some $\tau_0>0$ we have $|F(s)|<1$ for $s=\sigma_0+i\tau$ with $|\tau|\geq \tau_0$. 
\\
\noindent{\rm (ii)} There exist an increasing function $\varphi:\mathbb{R}\to\mathbb{R}$, a constant $K>0$ and sequences $\{T_m\}_{m\in\mathbb{N}}$, $\{T_m^{*}\}_{m\in\mathbb{N}}$ 

\,\,\,such that $\displaystyle\lim_{m\to\infty}T_m=\displaystyle\lim_{m\to\infty}T_m^{*}=\infty$,  
	\[
	T_m\leq T_{m+1} \leq \varphi(T_m), \hspace{0.3cm} 
	T_m^{*}\leq T_{m+1}^{*} \leq \varphi(T_m^{*}) \hspace{0.3cm} \text{for m} \in\mathbb{N},
	\]
	\,\,\,\,\,\,\,\,
	and $|F(s)|<e^{K|s|}$, for $s=\sigma+i\tau$ with  $0\leq\sigma\leq\sigma_0$ and $\tau=T_m$, $\tau=-T_m^{*}$, for $m\in\mathbb{N}$.\\
Then, for $T\geq 2$, we have that
\begin{equation} \label{llave}
	N(-\sigma_0, \sigma_0, -T, T, f)-N_{0}^{'}(T, f)\leq 4 \widehat{N}(0, \sigma_0, -\varphi(2T), \varphi(2T), h) +O(1),
	\end{equation} 
\end{theorem}

To prove that the function $h(s)$ defined in \eqref{0_45} satisfies the conditions of the previous theorem, we will use the estimates used by Ki. By \cite[Eq. (2.1)]{haseoki}, using the change of variable $s\mapsto -is-i/2$, we have that
\begin{align} \label{16_14_3_20}
h(s)=\Gamma(s-k)\Big(b_k\psi_{F,k}(s)+O\big(|s|^{-1/2}\big)\Big)
\end{align}
holds uniformly on the half-plane $\re{s}\geq-1/4$ and $|s|$ sufficiently large. On the other hand, by \cite[Theorem 2.1]{haseoki} it follows using the change of variable $s\mapsto -is+i/2$: for $\Delta>0$ sufficiently large,
		\begin{align} \label{16_14_3_21}
		\dfrac{h^*(-s)}{\Gamma(s-k-1)|\tau|^{\mu(\sigma)}}=O(1),
		\end{align}
for $s=\sigma+i\tau$ with $0\leq \sigma\leq \Delta$ and $|\tau|\geq 1$, and the function $\mu(\sigma)$ is given by
	\begin{equation*}
	\mu(\sigma) = \left\{
	\begin{array}{ll}
	1-\sigma, & \mathrm{si\ } 0\leq\sigma\leq1,
	\\
	0,             & \mathrm{si\ } \sigma>1.
	\end{array}
	\right.
	\end{equation*}
Finally, we will need to establish bounds for the right-hand side of \eqref{llave}, that implies to estimate the number of zeros of $h(s)$. The relation \eqref{16_14_3_20} tells us that we must study the behavior of the zeros of $\psi_F(s)$. We define $\psi_{F,k}(s):=\psi_F(s-k)$. Thus, using \eqref{17_36_2_20} this function we can be written as 
\begin{align}
\psi_{F,k}(s) = \displaystyle\sum_{m=0}^{n}a_me^{-\ln{((2m+1)\pi)(s-k)}}=e^{-\ln((2n+1)\pi)(s-k)}\Bigg[\displaystyle\sum_{m=0}^{n}p_me^{\beta_ms}\Bigg], \label{poly}
\end{align}
where $p_m=(a_{n-m})e^{-\beta_mk}$ and $\beta_m=\ln((2n+1)/(2(n-m)+1))$, for $0\leq m\leq n$. The sum on the right-hand side of \eqref{poly} is a Dirichlet polynomial in the framework \cite[Subsection 12.5]{bellman}. 

\begin{proposition} \label{19_4_3_20} Let $Z(\psi_{F,k})$ denote the set of zeros of $\psi_{F,k}(s)$. 
\begin{enumerate}
	\item There is a positive real number $c_0$ such that $Z(\psi_{F,k})\subset \{s\in\mathbb{C}: |\re{s}|<c_0\}$. 
	\item For $T_1<T_2$ and $c\geq c_0$, we have that 
	\[
	N(-c, c, T_1, T_2, \psi_{F,k})\leq n+\dfrac{\ln(2n+1)}{2\pi}(T_2-T_1).
	\]
	\item Let $K\subset\C$ such that $|\re{s}|\leq M$ for $s\in K$, and some $M>0$. Suppose that  \textit{K} is uniformly bounded from the zeros of $\psi_{F,k}(s)$, i.e.  
	$$\inf\{|s-z|:s\in \textit{K}, z\in Z(\psi_{F,k})\}>0.$$ Then, \,$\inf\{|\psi_{F,k}(s)|: s\in \textit{K}\}>0$.
\end{enumerate}	
\end{proposition}	 
\begin{proof} See \cite[Theorems 12.4, 12.5 and 12.6]{bellman}.
\end{proof}



\medskip

\section{Proofs of Theorem \ref{16_17_3_20} and Theorem \ref{teo1}}

\subsection{Proof of Theorem \ref{16_17_3_20}}

Let us define the function
\begin{align} \label{20_5_4_20}
F(s)= \frac{h^*(-s)}{h(s)}.
\end{align}
Since that $h(s)$ and $h^*(-s)$ are entire functions, we can choose $\sigma_0>0$ sufficiently large such that $F(s)\neq \pm 1$ and $h(s)\neq 0$ on the line $\re{s}=\sigma_0$. Using \eqref{16_14_3_20} and \eqref{16_14_3_21} we get for $s=\sigma+i\tau$ with $0\leq\sigma\leq\sigma_0$ and $|\tau|$ sufficiently large,  
	\begin{align}
	F(s) = \frac{O(1)\Gamma(s-k-1)|\tau|^{\mu(\sigma)}}{\Gamma(s-k)\big(b_k\psi_{F,k}(s)+O\big(|s|^{-1/2}\big)\big)}
	=\frac{O(1)|\tau|^{\mu(\sigma)}}{(s-k-1)\big(b_k\psi_{F,k}(s)+O\big(|s|^{-1/2}\big)\big)}. \label{7:32}
	\end{align}
Now, we analyze the behavior of $F(s)$ on the line $\re{s}=\sigma_0$. Note that $\mu(\sigma_0)=0$. On another hand, the line $\re{s}=\sigma_0$ is uniformly bounded from the zeros of $\psi_{F,k}(s)$. Then, recalling that $b_k\neq0$, by Proposition \ref{19_4_3_20} and the triangle inequality we get
\begin{align} \big|b_k\psi_{F,k}(s)+O\big(|s|^{-1/2}\big)\big|\gg 1, \label{triangulo}
\end{align}
for $s=\sigma_0+i\tau$, with $|\tau|$ sufficiently large,
Inserting this in \eqref{7:32}, it follows
	\begin{align*}
	|F(s)|\ll \frac{1}{|s-k-1|}. \label{cotafinal}
	\end{align*}
Therefore, for $s=\sigma_0+i\tau$ with $|\tau|$ sufficiently large we conclude that $|F(s)|<1$. This implies (i) of Theorem \ref{teoprin}. Let us to prove (ii) of Theorem \ref{teoprin}. For each $m\in\mathbb{Z}$ we consider the rectangle
	\[
	R_{m}=\{s\in\mathbb{C}:-\sigma_0< \re{s}< \sigma_0,\hspace{0.1cm} m< \im{s}< m+1\}.
	\]
	We divide this rectangle into $2n+1$ subrectangles $R_{m,j}$ defined by
	\[
	R_{m,j}=\Big\{s\in\mathbb{C}:-\sigma_0< \re{s}< \sigma_0, \hspace{0.1cm} m+\dfrac{j-1}{2n+1}< \im{s}< m+\dfrac{j}{2n+1}\Big\},
	\]
	for $j\in\{1,2,...,2n+1\}$. By Proposition \ref{19_4_3_20} we have that $N(-\sigma_0, \sigma_0, m, m+1, \psi_{F,k})\leq 2n$. So, there exists $j_0$ such that $\psi_{F,k}(s)$ does not vanish in $R_{m,j_0}$. Let us write
\[
	T_m=m+\dfrac{j_0-\frac{1}{2}}{2n+1}.
\]
	Note that $m<T_m<m+1$. Then, if we define $\varphi(x)=x+2$, we have that
	\begin{align*} 
	m<T_m<m+1<T_{m+1}<m+2<T_m+2=\varphi(T_m).
	\end{align*}
	Let $\textit{K}=\{s\in\mathbb{C}: -\sigma_0<\re{s}<\sigma_0, \hspace{0.1cm}  \im{s}=T_m, m\in\mathbb{Z}\}$. 
For any $s\in\textit{K}$, we have that $|s-z|\geq 1/{2(2n+1)}$, for all $z\in Z(\psi_{F,k})$. Then $\textit{K}$ is uniformly bounded from the zeros of $\psi_{F,k}(s)$. Using Proposition \ref{19_4_3_20} we see that \eqref{triangulo} holds for $s\in K$ with $|m|$ sufficiently large. Therefore, in \eqref{7:32} we obtain that for $s=\sigma+i\tau$ with $0\leq \sigma\leq \sigma_0$ and $\tau=T_m$ ($|m|$ sufficiently large) it follows
\[
F(s) \ll  \frac{|\tau|^{\mu(\sigma)}}{|s-k-1|}.
\]
Using the fact that $\mu(\sigma)\leq 1$, we conclude that 
		\begin{align*}
			|F(s)|\ll 1<e^{|s|}.
		\end{align*}
Now, we choose $T_m^{*}=-T_{-m}$, for all $m\in\mathbb{N}$. Thus, we obtain (ii) of Theorem \ref{teoprin}. Therefore
	\begin{equation} \label{compa}
N(-\sigma_0, \sigma_0, -T, T, C_F)-N_{0}^{'}(T, C_F)\leq 4 \widehat{N}(0, \sigma_0, -\varphi(2T), \varphi(2T), h) +O(1).
\end{equation}
To conclude we need to bound $\widehat{N}(0, \sigma_0, -\varphi(2T), \varphi(2T), h)$. Firstly, we choose $0<\varepsilon<1/4$ such that $h(s)$ and $\psi_{F,k}(s)$ do not vanish on $\re{s}=-\varepsilon_0$. The definition of $T_m$ implies that
\begin{align} \label{19_44_4_02}
\dfrac{1}{2n+1}\leq T_{m+1}-T_m\leq 2,
\end{align}
and using Proposition \ref{19_4_3_20} we obtain $N(-\varepsilon,\sigma_0,T_m,T_{m+1},\psi_{F,k})\leq 2n$. Let us to divide the rectangle $\{s\in\C: -\varepsilon< \re{s}< 0 \hspace{0.1cm} \mbox{and} \hspace{0.1cm} T_m < \im{s} <T_{m+1}\}$ into $2n+1$ vertical subrectangles with horizontal length $\varepsilon/(2n+1)$. So, one of this rectangles, denoted by $I_m$, has no zeros of $\psi_{F,k}(s)$ and $h(s)$. Suppose that the right vertical side of $I_m$ is contained on the line $\re{s}=-\varepsilon_{m}$, that we can suppose without loss of generality that doesn't contain a zero of $\psi_{F,k}(s)$. Now, if we place a circle of radius $\delta>0$ sufficiently small (for instance $\delta<1/(2n+1)(16n)$) we can enclosed the zeros of the rectangle $J_m=\{s\in\C: -\varepsilon_m< \re{s}< \sigma_0 \hspace{0.1cm} \mbox{and} \hspace{0.1cm} T_m< \im{s}< T_{m+1}\}$ in a contour $C_m$ such that the distance between $C_m$ and $J_m$ is at least $1/(2n+1)(16n)$ and $C_m$ is distanced at least $1/(2n+1)(32n)$ from the zeros of $\psi_{F,k}(s)$. Therefore, the union of the  contour $C_m$ for all $m\in\Z$ is uniformly bounded from the zeros. By Proposition \ref{19_4_3_20} there is a constant $M>0$ such that $|\psi_{F,k}(s)|>M/|b_k|$ for each $s\in C_m$. Using \eqref{16_14_3_20} we get that 
$$
\bigg|b_k\psi_{F,k}(s)-\dfrac{h(s)}{\Gamma(s-k)}\bigg|<M<|b_k\psi_{F,k}(s)|
$$
for $s\in C_m$, with $|m|$ sufficiently large. If we denote $w(s)=h(s)/\Gamma(s-k)$, applying Rouch\'e's theorem we obtain that there is $m_0\in\mathbb{N}$ sufficiently large such that
\begin{align} N(-\varepsilon_m, \sigma_0,T_m,T_{m+1},w)=N(-\varepsilon_m,\sigma_0,T_m,T_{m+1},\psi_{F,k}),
\label{hola1} 
\end{align}
and
\begin{align}
N(-\varepsilon_{-m-1}, \sigma_0, T_{-m-1},T_{-m},w)=N(-\varepsilon_{-m-1},\sigma_0,T_{-m-1},T_{-m},\psi_{F,k}) \label{hola2} 
\end{align}
for $m\geq m_0$. On another hand, by analyticity of $h(s)$ we have 
\begin{align} \label{25}
N(-1/4,\sigma_0,-T_{m_0}-1,T_{m_0}+1,h)=O(1).
\end{align}
Finally, let $T$ be a positive real parameter. If $T< T_{m_0}$, by \eqref{25} we obtain $N(0,\sigma_0,0,T,h)=O(1)$. If $T\geq T_{m_0}$, we choose $m_1\geq m_0\geq 1$ such that $m_1<T_{m_1}\leq T<T_{m_1+1}<m_1+2$. Since that the zeros of $1/\Gamma(s)$ are the non-positive integers, by \eqref{hola1} , \eqref{25}, Proposition \ref{19_4_3_20} and \eqref{19_44_4_02}, we get
\begin{align*}
\widehat{N}(0,\sigma_0,0,T,h)& \leq \displaystyle\sum_{j=m_0}^{m_1}N(-\varepsilon_j,\sigma_0,T_{j},T_{j+1},h) + \widehat{N}(0,\sigma_0,0,T_{m_0}+1,h) \nonumber \\
& = \displaystyle\sum_{j=m_0}^{m_1}N(-\varepsilon_j,\sigma_0,T_{j},T_{j+1},w) +O(1)=  \displaystyle\sum_{j=m_0}^{m_1}N(-\varepsilon_0,\sigma_0,T_{j},T_{j+1},\psi_{F,k})+O(1) \nonumber \\
& \leq \displaystyle\sum_{j=m_0}^{m_1}\bigg(n+\dfrac{\ln(2n+1)}{2\pi}(T_{j+1}-T_{j})\bigg)+O(1)  \leq  \bigg(n+\dfrac{\ln(2n+1)}{\pi}\bigg)T +O(1). 
\end{align*}
Similarly,  for $T<0$ we use \eqref{hola2} to obtain a similar bound. Thus, we obtain for $T>0$ that  
\[
\widehat{N}(0,\sigma_0,-T,T,h) \leq \bigg(2n+\dfrac{2\ln(2n+1)}{\pi}\bigg)T +O(1).
\]
We replace $T$ by $\varphi(2T)$ in the above expression, and inserting in \eqref{compa}, and one can see that
\begin{align} \label{17_29_5_20}
N(-\sigma_0, \sigma_0, -T, T, C_F)-N_{0}^{'}(T, C_F)\leq \bigg(16n+\dfrac{16\ln(2n+1)}{\pi}\bigg)T +O(1).
\end{align}
To obtain our desired result we will use an argument of Ki in \cite[Pag. 131]{haseoki}. Following his idea, for $T>0$ we get that
\begin{align} \label{17_29_5_21}
N(-\sigma_0, \sigma_0, -T, T, C_F) - \overline{N_{1}}(T, C_F) \leq 2\Bigg(N(-\sigma_0, \sigma_0, 0, T, C_F) - \displaystyle\sum_{k=1}^\infty \overline{N_{k}}(T, C_F)\Bigg),
\end{align}
where $\overline{N_k}(T,C_F)$ denotes the number of zeros of $C_F$ with multiplicity $k$ with $|\im{s}|< T$ and $\re{s}=0$, counting with multiplicity. Note that 
\begin{align} \label{17_29_5_23}
N_{0}^{'}(T, C_F) \leq \displaystyle\sum_{k=1}^\infty \overline{N_{k}}(T, C_F).
\end{align}
We conclude combining \eqref{17_29_5_20}, \eqref{17_29_5_21}, \eqref{17_29_5_23}, and recalling by \eqref{18_58_3_20} that $
\overline{N}(T,C_F)=N(-\sigma_0, \sigma_0, -T ,T , C_F)$.

\smallskip
\subsection{Proof of Theorem \ref{teo1}} The proof is similar as the previous case. Using the function defined in \eqref{20_5_4_20}, without loss of generality we can choose $\delta>0$ in such a way that $F(s)\neq \pm1$, $h(s)\neq 0$ and $C_F(s)\neq 0$ when $\sigma=\delta$. By \eqref{16_14_3_20} and \eqref{16_14_3_21} it follows for $s=\sigma+i\tau$ with $0\leq\sigma\leq\delta$ and $|\tau|$ sufficiently large
\begin{align*}
F(s) =\frac{O(1)|\tau|^{\mu(\sigma)}}{(s-k-1)\big(b_k\psi_{F,k}(s)+O\big(|s|^{-1/2}\big)\big)}.
\end{align*}
Using the fact that the $\psi_{F,k}(s)$ has no zeros in the strip $-\Delta_{**}<\re{s}<\Delta_{*}$, by Proposition \ref{19_4_3_20} we get
\begin{align} \label{17_47_5_20}
\Big|b_k\psi(s-k)+O\big(|s|^{-1/2}\big)\Big|\gg 1,
\end{align}
for $s=\sigma+i\tau$, with $0\leq\sigma\leq\delta$ and $|\tau|$ sufficiently large. Therefore
\begin{align}
|F(s)|\ll \frac{|\tau|^{\mu(\sigma)}}{|s-k-1|}. \label{cotafinal2}
\end{align}
Using the fact that $\mu(\delta)<1$, then
\[
|F(s)|\ll \frac{|\tau|^{\mu(\delta)}}{|s-k-1|}\ll \frac{1}{|\tau|^{1-\mu(\delta)}}<1,
\]
for $s=\delta+i\tau$, with $|\tau|$ sufficiently large. Further, we have that $\mu(\sigma)\leq1$, which implies in \eqref{cotafinal2} that
\begin{align*}
		|F(s)|\ll \frac{|\tau|^{\mu(\sigma)}}{|s-k-1|}\ll 1<e^{|s|},
\end{align*}
for $s=\sigma+i\tau$ with $0\leq\sigma\leq\delta$ and $|\tau|$ sufficiently large. Choosing $\varphi(x)=x+2$ and $T_m=T^*_{m}=m$, for $m$ sufficiently large, we get that the hypotheses in Theorem \ref{teoprin} are satisfied. Then
\begin{align} \label{19_20_5_20}
N(-\delta, \delta, -T, T, C_F)-N_{0}^{'}(T, C_F)\leq 4 \widehat{N}(0, \delta, -\varphi(2T), \varphi(2T), h) +O(1).
\end{align}
Combining \eqref{16_14_3_20} and \eqref{17_47_5_20}, we get a constant $L>0$ such that $
|h(s)|\geq L|\Gamma(s-k)|
$ for $s=\sigma+i\tau$ with $0\leq\sigma\leq\delta$ and $|\tau|$ sufficiently large. Then, $h(s)$ only has finitely many zeros on the strip $0\leq\sigma\leq\delta$, because all possible zeros are contained in a compact set. Therefore, the right-hand side in \eqref{19_20_5_20} is bounded and this implies our desired result.

\section*{Acknowledgements}
	A. C. was supported by Grant 227768 of the Research Council of Norway. Part of the project was completed during my stay at IMCA with excellent working conditions.

\end{document}